\pdfoutput=1
\RequirePackage{ifpdf}
\ifpdf 
\documentclass[pdftex]{sigma}
\else
\documentclass{sigma}
\fi

\numberwithin{equation}{section}

\newtheorem{Theorem}{Theorem}[section]
\newtheorem*{Theorem*}{Theorem}
\newtheorem{Corollary}[Theorem]{Corollary}
\newtheorem{Lemma}[Theorem]{Lemma}
\newtheorem{Proposition}[Theorem]{Proposition}
 { \theoremstyle{definition}
\newtheorem{Definition}[Theorem]{Definition}

\newtheorem{Example}[Theorem]{Example}
\newtheorem{Remark}[Theorem]{Remark} }

\makeatletter
\newcommand*\owedge{\mathpalette\@owedge\relax}
\newcommand*\@owedge[1]{%
 \mathbin{%
 \ooalign{%
 $#1\m@th\bigcirc$\cr
 \hidewidth$#1\m@th\wedge$\hidewidth\cr
 }%
 }%
}
\makeatother

\begin{document}
\allowdisplaybreaks

\renewcommand{\PaperNumber}{087}

\FirstPageHeading

\ShortArticleName{Deformation of the Weighted Scalar Curvature}

\ArticleName{Deformation of the Weighted Scalar Curvature}

\Author{Pak Tung HO~$^{\rm a}$ and Jinwoo SHIN~$^{\rm b}$}

\AuthorNameForHeading{P.T.~Ho and J.~Shin}

\Address{$^{\rm a)}$~Department of Mathematics, Tamkang University, Tamsui, New Taipei City 251301, Taiwan}
\EmailD{\href{mailto:paktungho@yahoo.com.hk}{paktungho@yahoo.com.hk}}

\Address{$^{\rm b)}$~Korea Institute for Advanced Study, Hoegiro 85, Seoul 02455, Korea}
\EmailD{\href{mailto:shinjin@kias.re.kr}{shinjin@kias.re.kr}}

\ArticleDates{Received December 07, 2022, in final form October 30, 2023; Published online November 04, 2023}

\Abstract{Inspired by the work of Fischer--Marsden [\textit{Duke Math.~J.} \textbf{42} (1975), 519--547], we study in this paper the deformation of the weighted scalar curvature. By studying the kernel of the formal $L_\phi^2$-adjoint for the linearization of the weighted scalar curvature, we prove several geometric results. In particular, we define a~weighted vacuum static space, and study locally conformally flat weighted vacuum static spaces. We then prove some stability results of the weighted scalar curvature on flat spaces. Finally, we consider the prescribed weighted scalar curvature problem on closed smooth metric measure spaces.}


\Classification{53C21; 53C23}

\section{Introduction}

A \textit{smooth metric measure space} is the tuple $\bigl(M,g,{\rm e}^{-\phi}{\rm d}V_g,m\bigr)$, where $(M,g)$ is a smooth Riemannian manifold, ${\rm d}V_g$ is the volume form of $g$, ${\rm e}^{-\phi}{\rm d}V_g$ is a smooth measure determined by $\phi\in C^\infty(M)$, and $m$ is a dimensional parameter with $0\leq m\leq \infty$. It was first introduced by Bakry and \'{E}mery in \cite{Bakry&Emery}. Smooth metric measure spaces have recently attracted a lot of attention in Riemannian geometry. For example, they play an important role in Perelman's approach to the Ricci flow \cite{Perelman}.
The \textit{weighted scalar curvature} of the smooth metric measure space $\bigl(M,g,{\rm e}^{-\phi}{\rm d}V_g,m\bigr)$ is defined as
\begin{equation}\label{wsc}
 R_\phi^m=R+2\Delta \phi-\frac{m+1}{m}|\nabla \phi|^2,
\end{equation}
where $R$ is the scalar curvature of $g$, $\Delta$ and $\nabla$ are the Laplacian and the gradient of $g$, respectively. The geometric meaning of the weighted scalar curvature is as follows: Let $(F^m,h)$ be the flat $m$-torus. We regard $\bigl(M,g,{\rm e}^{-\phi}{\rm d}V_g,m\bigr)$ as the base of the warped product
\begin{equation}\label{warped1}
 \bigl(M\times F^m,g\oplus {\rm e}^{-\frac{2\phi}{m}}h\bigr).
\end{equation}
Then the weighted scalar curvature $R_\phi^m$ is the scalar curvature of the warped product \eqref{warped1}.

In \cite{Fischer&Marsden}, Fischer and Marsden studied
the deformation of the scalar curvature. In order to study this problem, they studied the kernel of formal $L^2$-adjoint for the linearization of scalar curvature. More precisely, they considered the scalar curvature as a function on the space of Riemannian metrics, i.e.,
the scalar curvature map $g\mapsto R(g)$.
Then they computed the linearization
\[L_g(h):= \left.\frac{{\rm d}}{{\rm d}t}\right|_{t=0}R(g+th)\]
of the scalar curvature map and its formal $L^2$-adjoint~$L_g^*$. They showed that a closed Riemannian manifold with $\ker L_g^*= \{0\}$ is linearization stable and hence any smooth function sufficiently close to the scalar curvature of the background metric can be realized as the scalar curvature of a nearby metric.
A Riemannian manifold $(M,g)$ with $\ker L_g^*\neq \{0\}$ is also called \textit{vacuum static space}, which has been widely studied by many authors. See \cite{Corvino1,Corvino&Eichmair&Miao, Corvino&Schoen, Kobayashi, Qing&Yuan1, Qing&Yuan2} and the references therein.
In recent years there has been developed a general theory for ``conformally variational invariants''
which gives sufficient conditions to perturb a given Riemannian manifold such that a given scalar
Riemannian invariant achieves a given function \cite{Case&Lin&Yuan1,Case&Lin&Yuan2}.
Specializing this to scalar curvature recovers the result of Fischer--Marsden, and the results are
also applicable to $Q$-curvature, $\sigma_2$-curvature, and many more invariants.

Inspired by \cite{Fischer&Marsden}, we study in this paper the deformation of the weighted scalar curvature.
To~do this, we regard the weighted scalar curvature $R_\phi^m$ as a
function on the space of Riemannian metrics and the space of smooth functions,
\begin{equation*}
 (g,\phi)\mapsto \mathcal{R}(g,\phi)=R+2\Delta \phi -\frac{m+1}{m}|\nabla \phi|^2.
\end{equation*}
Let $\mathcal{DR}_{g,\phi}$ be the linearization of this weighted scalar curvature map, and let $\mathcal{DR}_{g,\phi}^*$ be the formal $L^2$-adjoint of $\mathcal{DR}_{g,\phi}$ in the weighted sense.
In Section \ref{sec3}, we find the precise expression of $\mathcal{DR}_{g,\phi}$ and $\mathcal{DR}^*_{g,\phi}$.
Similar to the vacuum static space, the \textit{weighted vacuum static space}
is defined to~be the smooth metric measure space with $\ker \mathcal{DR}_{g,\phi}^*\neq \{0\}$ (see Definition \ref{def_wstatic}).
By analyzing the kernel of $\mathcal{DR}_{g,\phi}^*$, we prove several geometric results.
In particular, we prove the following classification result of closed weighted vacuum static spaces.

\begin{Proposition}\label{propRneg}
 Let $\bigl(M,g,{\rm e}^{-\phi}{\rm d}V_g,m\bigr)$ be a closed weighted vacuum static space. Then $\bigl(M,g,\allowbreak {\rm e}^{-\phi}{\rm d}V_g,m\bigr)$ is either isometric to a Ricci-flat manifold with $\phi$
 being constant, or the weighted scalar curvature $R_\phi^m$ is a positive constant.
 In particular, any weighted vacuum static space has nonnegative constant weighted scalar curvature.
\end{Proposition}

In Section \ref{sec4}, we prove the following.

\begin{Theorem}\label{thm1.1}
 Let $\bigl(M,g,{\rm e}^{-\phi}{\rm d}V_g,m\bigr)$ be a connected weighted vacuum static space of dimension $n\geq3$ $($not necessarily compact$)$. If $\bigl(M,g,{\rm e}^{-\phi}{\rm d}V_g,m\bigr)$ is locally conformally flat in the weighted sense, then around any regular point of $f$, the manifold $(M,g)$ is locally a warped product with $(n-1)$-dimensional fibers of constant sectional curvature.
 \end{Theorem}
We say that $\bigl(M,g,{\rm e}^{-\phi}{\rm d}V_g,m\bigr)$ is \textit{locally conformally flat in the weighted sense} if for each $p\in M$, there is a neighborhood $U$ of $p$ on which there is a conformal factor $u$ for which $\bigl({\rm e}^{2u}g,{\rm e}^{(m+n)u}{\rm e}^{-\phi}{\rm d}V_g\bigr)=\bigl(g_{\rm flat},{\rm d}V_{g_{\rm flat}}\bigr)$. In \cite{Kobayashi}, Kobayashi classified the conformally flat vacuum static spaces. Independently, Lafontaine \cite{Lafontaine} obtained a classification of closed conformally flat vacuum static spaces.

 Linearization stability of the scalar curvature on flat spaces was studied in \cite{Fischer&Marsden},
 and the corresponding stability for $Q$-curvature was studied in \cite{Lin&Yuan}.
 Inspired by these results, in Section~\ref{sec5}, we consider the linearization stability of the weighted scalar curvature on flat spaces. In particular, we have the following.

\begin{Theorem}\label{rigidity}
 For $n\geq3$, let $\bigl(M,\overline{g},{\rm e}^{-\overline{\phi}}{\rm d}V_{\overline{g}},m\bigr)$ be a compact smooth metric measure space with
 $\overline{g}$ being flat, $\overline{\phi}$ being constant, and $m$ being a positive integer.
 There is an $\epsilon>0$ such that if $(g,\phi)$ has nonnegative weighted scalar curvature and is $\epsilon$-close to $\bigl(\overline{g},\overline{\phi}\bigr)$ in $C^2$, then $g$ is flat and $\phi$ is constant.
\end{Theorem}

In Section \ref{sec5}, we also consider the prescribed weighted scalar curvature problem on closed smooth metric measure spaces.
We have the following.
\begin{Theorem}\label{thm1.3}
 Let $\bigl(M,g_0,{\rm e}^{-\phi_0}{\rm d}V_{g_0},m\bigr)$ be a closed smooth metric measure space with weighted scalar curvature $R_{\phi_0}^m$, and let $K$ be a non-constant smooth function on $M$. If there is a constant $c_0>0$ satisfying
\begin{equation*}
c_0\min K<R_{\phi_0}^m<c_0\max K,
\end{equation*}
there is a smooth metric measure $\bigl(M,g,{\rm e}^{-\phi}{\rm d}V_g,m\bigr)$ such that its weighted scalar curvature is~$K$.
\end{Theorem}

We remark that, on closed manifolds, the prescribed scalar curvature problem has been studied by Kazdan and Warner in \cite{Kazdan&Warner1, Kazdan&Warner4, Kazdan&Warner2, Kazdan&Warner3},
and on compact Riemannian manifolds with boundary, the problem of prescribing the scalar curvature in $M$ and the mean curvature on the boundary $\partial M$ was studied in \cite{Cruz&Vitorio, Ho&Huang}.

 After this paper was finished,
it was brought to our attention that many more relevant weighted invariants are
constructed through the recent work of Khaitan \cite{Khaitan1,Khaitan2}. In light of the results of the $\sigma_2$-invariant and renormalized volume coefficients in closed Riemannian manifolds, we expect that results similar to ours also hold for the invariant constructed by Khaitan (cf.~\cite{Andrade,Silva1}).\looseness=-1

\section{Smooth metric measure space}

We collect in this section some basic definitions and facts for smooth metric measure spaces
which will be needed in the rest of the paper.
First, we have the following.

\begin{Definition}
 A \textit{smooth metric measure space} is a four-tuple $\bigl(M,g,{\rm e}^{-\phi}{\rm d}V_g,m\bigr)$,
 where $(M,g)$ is an $n$-dimensional Riemannian manifold, $\phi$ is a smooth function in $M$, and $m\in(0,\infty)$ is a~dimensional parameter.
\end{Definition}

We remark that in this paper we consider only the case where $m$ is a positive finite number. When $m=\infty$, the weighted scalar curvature is sometimes called the \textit{$P$-scalar curvature} and has been studied in \cite{Abedin&Corvino}.

\begin{Definition}
 Let $\bigl(M,g,{\rm e}^{-\phi}{\rm d}V_g,m\bigr)$ be a smooth metric measure space and let $(V,h_V)$ and $(W,h_W)$ be vector bundles with inner product over $M$, and let $\langle\cdot,\cdot\rangle_V$ and $\langle \cdot,\cdot\rangle_W$ be the corresponding inner products
 \begin{equation*}
 \langle \zeta_1,\zeta_2\rangle_V=\int_M h_V(\zeta_1,\zeta_2){\rm e}^{-\phi}{\rm d}V_g,\qquad
 \langle \xi_1,\xi_2\rangle_W=\int_M h_W(\xi_1,\xi_2){\rm e}^{-\phi}{\rm d}V_g
 \end{equation*}
 on sections $\zeta_i\in \Gamma(V)$ and $\xi_i\in\Gamma(W)$ determined by the measure ${\rm e}^{-\phi}{\rm d}V_g$. The \textit{weighted divergence} ${\rm div}_\phi\colon\Gamma(W)\to \Gamma(V)$ of an operator $D\colon\Gamma(V)\to \Gamma(W)$ is the (negative of the) formal adjoint of $D$ with respect to the inner products $\langle\cdot,\cdot\rangle_V$ and $\langle\cdot,\cdot\rangle_W$, i.e., for all $\zeta\in \Gamma(V)$ and $\xi\in \Gamma(W)$, at least one of which is compactly supported in $M$,
 \begin{equation*}
 \langle D(\zeta),\xi\rangle_W=-\langle \zeta,{\rm div}_\phi\xi\rangle_V.
 \end{equation*}
 The \textit{weighted Laplacian} $\Delta_\phi\colon C^\infty(M)\to C^\infty(M)$ is the operator $\Delta_\phi={\rm div}_\phi d$.
\end{Definition}

\begin{Lemma}[{\cite[Lemma 3.5]{Case4}}]
Let $\bigl(M,g,{\rm e}^{-\phi}{\rm d}V_g,m\bigr)$ be a smooth metric measure space.
The weighted divergence ${\rm div}_\phi$ is related to the usual divergence ${\rm div}_g$ by
\begin{equation*}
 {\rm div}_\phi={\rm e}^{\phi}\circ {\rm div}_g\circ {\rm e}^{-\phi},
\end{equation*}
where ${\rm e}^\phi$ and ${\rm e}^{-\phi}$ are regarded as multiplication operators. In particular, we have the formulas%
 \begin{equation*}
 {\rm div}_\phi \omega={\rm div}_g\omega-\iota_{\nabla \phi}\omega,\qquad
 \Delta_\phi w=\Delta_gw-\langle \nabla \phi,\nabla w\rangle
 \end{equation*}
 for all $\omega\in \Lambda^kT^*M$ and all $w\in C^\infty(M)$.
\end{Lemma}

When $M$ is closed, it is well known that
\begin{equation}
\int_M \langle \nabla f,X\rangle {\rm e}^{-\phi}{\rm d}V_g=-\int_M f {\rm div}_\phi(X){\rm e}^{-\phi}{\rm d}V_g \label{IBP}
\end{equation}
for any vector field $X$ in $M$.

\begin{Definition}
 Let $\bigl(M,g,{\rm e}^{-\phi}{\rm d}V_g,m\bigr)$ be a smooth metric measure space. The \textit{Bakry--\'{E}mery Ricci tensor} ${\rm Rc}_\phi^m$ is the symmetric $(0,2)$-tensor
 \begin{equation*}
 {\rm Rc}_\phi^m:={\rm Rc}+{\rm Hess}_g\phi-\frac{1}{m} {\rm d}\phi\otimes {\rm d}\phi,
 \end{equation*}
 where ${\rm Rc}$ is the Ricci tensor with respect to $g$, and ${\rm Hess}_g$ is the Hessian of $\phi$ with respect to $g$.
 The \textit{weighted scalar curvature} $R_\phi^m$ is defined as
 \begin{equation*}
 R_\phi^m:=R+2\Delta_g\phi-\frac{m+1}{m} |\nabla \phi |^2,
 \end{equation*}
 where $R$ is the scalar curvature with respect to $g$.
\end{Definition}
It is important to note that the weighted scalar curvature is in general not the trace of the Bakry--\'{E}mery Ricci tensor. Indeed,
there holds (see \cite[formula~(4.1)]{Case4})
\begin{equation*}
 R_\phi^m=\operatorname{tr}{\rm Rc}_\phi^m+\Delta_\phi\phi.
\end{equation*}
The following proposition was proved in \cite[Proposition 4.2]{Case4}.

\begin{Lemma}\label{bianchi}
There holds
 \begin{equation*}
 {\rm div}_\phi {\rm Rc}_\phi^m=\frac{1}{2}{\rm d}R_\phi^m-\frac{1}{m}\Delta_\phi\phi {\rm d}\phi.
 \end{equation*}
\end{Lemma}

\section{Deformation of the weighted Scalar curvature}\label{sec3}

Throughout this section, we will always assume $\bigl(M,g,{\rm e}^{-\phi}{\rm d}V_g,m\bigr)$ is an $n$-dimensional closed smooth metric measure space $(n\geq3)$ unless otherwise stated.
We study the \textit{weighted scalar curvature map}
\[(g,\phi)\mapsto \mathcal{R}(g,\phi):=R+2\Delta \phi -\frac{m+1}{m} |\nabla \phi |^2.\]
We first compute the linearization of this map and its adjoint.
To this end, we let $\mathcal{M}$ be the space of Riemannian metrics on $M$.
The following formula is known.

\begin{Lemma}\label{lem3.1}
Let $k$ be a nonnegative integer, and let $l>\frac{n}{p}+2$. The map $\mathcal{R}$ as a map $\mathcal{R}\colon\bigl(C^{k+2,\alpha}\cap \mathcal{M}\bigr)\times C^{k+2,\alpha}\to C^{k,\alpha}$, or $\mathcal{R}\colon\bigl(W^{l,p}\cap \mathcal{M}\bigr)\times W^{l,p}\to W^{l-2,p}$ is smooth. The linearization of $R_\phi^m$ is given by
\begin{gather*}
\mathcal{DR}_{g,\phi}(h,\psi):=\left.\frac{{\rm d}}{{\rm d}t}\right|_{t=0}R_{\phi+t\psi}^m(g+th)\\
\hphantom{\mathcal{DR}_{g,\phi}(h,\psi):}={\rm div}_\phi {\rm div}_\phi h-\bigr\langle h,{\rm Rc}_\phi^m\bigl\rangle -\Delta_\phi({\rm tr}_gh)+2\left(\Delta_\phi \psi-\frac{1}{m}\langle {\rm d}\phi,{\rm d}\psi\rangle\right).
\end{gather*}
\end{Lemma}
\begin{proof}
See, for example, \cite[Proposition 5.1]{Case2}.
\end{proof}

The formal $L^2_\phi$-adjoint $\mathcal{DR}_{g,\phi}^*$ of $\mathcal{DR}_{g,\phi}$ is defined as
\[
\int_M f \mathcal{DR}_{g,\phi}(h,\psi) {\rm e}^{-\phi}{\rm d}V_g=\int_M \bigl\langle \mathcal{DR}_{g,\phi}^*(f), (h,\psi)\bigr\rangle {\rm e}^{-\phi}{\rm d}V_g.
\]
Using \eqref{IBP} and Lemma \ref{lem3.1}, we can compute $\mathcal{DR}_{g,\phi}^*$ as follows:
\begin{equation*}
 \begin{split}
 \mathcal{DR}^*_{g,\phi}(f)=\left(-(\Delta_\phi f)g+{\rm Hess}_gf-f{\rm Rc}_\phi^m,2\Delta_\phi f+\frac{2}{m}\bigl(\langle {\rm d}f,{\rm d}\phi\rangle+f\Delta_\phi\phi\bigr)\right).
 \end{split}
\end{equation*}
Note that the adjoint operator $\mathcal{DR}_{g,\phi}^*$ is overdetermined-elliptic, i.e., it has injective symbol. Indeed, the principal symbol of $\mathcal{DR}^*$ is given by
\begin{equation*}
 \sigma(\mathcal{DR}_{g,\phi}^*)_x(\epsilon)f=\left(\bigl(\|\epsilon\|^2g-\epsilon\otimes \epsilon\bigr)f,-\frac{2}{m}\|\epsilon\|^2f\right).
\end{equation*}
Then one can easily see that it is injective for $\epsilon\neq0$. Since $\mathcal{DR}_{(g,\phi)}^*$ is overdetermined-elliptic, the operator $\mathcal{DR}_{(g,\phi)}\mathcal{DR}^*_{(g,\phi)}$ is strictly elliptic. Thus the kernel of this operator is composed of smooth fields (assuming that $g$ and $\phi$ are smooth) by elliptic regularity, and on a closed manifold, this kernel is the same as the kernel of the adjoint $\mathcal{DR}_{(g,\phi)}^*$, via the identity
\[
\int_M f\mathcal{DR}_{(g,\phi)}\mathcal{DR}_{(g,\phi)}^*(f){\rm e}^{-\phi}{\rm d}V_g=\int_M \bigl|\mathcal{DR}_{(g,\phi)}^*(f)\bigr|^2{\rm e}^{-\phi}{\rm d}V_g.
\]
 A similar computation shows that the image of $\mathcal{DR}_{(g,\phi)}$ is $L^2_\phi$-orthogonal to the kernel of $\mathcal{DR}_{(g,\phi)}^*$.

\begin{Definition}\label{def_wstatic}
We say that the smooth metric measure space (not necessarily compact) $\bigl(M,g,{\rm e}^{-\phi}{\rm d}V_g,m\bigr)$ is a \textit{weighted vacuum static space} if there exists a smooth function $f\not\equiv 0$ such that
\begin{equation}\label{wstatic}
 -(\Delta_\phi f)g+{\rm Hess}_gf-f{\rm Rc}_\phi^m=0,\qquad\textrm{ and } \qquad \Delta_\phi f+\frac{1}{m}\bigl(\langle {\rm d}f,{\rm d}\phi\rangle+f\Delta_\phi\phi\bigr)=0.
\end{equation}
We say that $\bigl(M,g,{\rm e}^{-\phi}{\rm d}V_g,m\bigr)$ is a \textit{nontrivial weighted vacuum static space}
if there exists $f$ satisfying (\ref{wstatic}) which is not constant.
\end{Definition}

Note that $f$ satisfies (\ref{wstatic}) if and only if
$f$ lies in the kernel of $\mathcal{DR}^*_{g,\phi}$.
Hence, $\bigl(M,g,{\rm e}^{-\phi}{\rm d}V_g,m\bigr)$ is a weighted vacuum static space
if and only if the kernel of $\mathcal{DR}^*_{g,\phi}$ is nontrivial.

Taking the trace of the first equation in \eqref{wstatic} yields
\begin{equation*}
 (1-n)\Delta_\phi f+f\Delta_\phi\phi+\langle \nabla f,\nabla \phi\rangle -fR_\phi^m=0.
\end{equation*}
Combining this with the second equation in \eqref{wstatic}, we have
\begin{equation}\label{wstatic_del}
 \Delta_\phi f=-\frac{f R_\phi^m}{n+m-1}.
\end{equation}
So the first equation in \eqref{wstatic} can be written as
\begin{equation}\label{wstatic2}
{\rm Hess}_gf=f\left({\rm Rc}_\phi^m-\frac{R_\phi^m}{n+m-1}g\right).
\end{equation}

By analyzing the equations satisfied by a weighted vacuum static space,
we obtain some results. First, we have the following.

\begin{Proposition}\label{Rphiconst}
 Let $\bigl(M,g,{\rm e}^{-\phi}{\rm d}V_g,m\bigr)$ be a connected closed weighted vacuum static space. Then
 the weighted scalar curvature $R_\phi^m$ must be constant.
\end{Proposition}
\begin{proof}
 Note that if $f$ is a non-zero constant, then there is nothing to prove. So we assume that $\bigl(M,g,{\rm e}^{-\phi}{\rm d}V_g,m\bigr)$ is nontrivial. By direct calculation and using Lemma \ref{bianchi}, we get the following:
 \begin{gather*}
 {\rm div}_\phi(\Delta_\phi fg)_i=\nabla_i\Delta_\phi f-(\Delta_\phi f)\nabla_i \phi,\\
 {\rm div}_\phi({\rm Hess}_gf)_i=\nabla_i\Delta_\phi f+\bigl({\rm Rc}_\phi^m\bigr)_{il}\nabla^lf+\frac{1}{m}\langle \nabla \phi,\nabla f\rangle \nabla_i\phi,\\
 {\rm div}_\phi\bigl(f{\rm Rc}_\phi^m\bigr)_i=\bigl({\rm Rc}_\phi^m\bigr)_{il}\nabla^lf+\frac{1}{2}f \nabla_iR_\phi^m-\frac{1}{m}f(\Delta_\phi\phi) \nabla_i\phi.
 \end{gather*}
Thus, taking the weighted divergence of the first equation in \eqref{wstatic} yields
\begin{gather*}
 {\rm div}_\phi\bigl(-(\Delta_\phi f)g+{\rm Hess}_gf-f{\rm Rc}_\phi^m\bigr)_i
 =\left[\Delta_\phi f+\frac{1}{m}\bigl(\langle \nabla \phi,\nabla f\rangle+f\Delta_\phi\phi\bigr)\right]\nabla_i\phi-\frac{1}{2}f\nabla_iR_\phi^m.
\end{gather*}
Then, by the second equation in \eqref{wstatic}, it reduces to
\begin{equation*}
 f{\rm d}R_\phi^m=0.
\end{equation*}
If $f$ is never zero, $R_\phi^m$ must be constant. On the other hand, assume there is some $x_0\in M$ with $f(x_0)=0$. Then we must have ${\rm d}f(x_0)\neq0$. To see this, assume ${\rm d}f(x_0)=0$, let $\gamma(t)$ be a geodesic starting at $x$, and let $h(t)=f\bigl(\gamma(t)\bigr)$. It follows from \eqref{wstatic2}
that $h(t)$ satisfies the following linear second order differential equation
\begin{gather*}
 h''(t)=({\rm Hess}_gf)_{\gamma(t)}\cdot\bigl(\gamma'(t),\gamma'(t)\bigr)
 =\left\{\left({\rm Rc}_\phi^m-\frac{R_\phi^m}{n+m-1}g\right)\cdot\bigl(\gamma'(t),\gamma'(t)\bigr)\right\}h(t)
\end{gather*}
with $h(0)=f(x_0)=0$ and $h'(0)={\rm d}f(\gamma(0))\cdot \gamma'(0)=0$. This implies that $h(t)\equiv 0$. Thus, $f$~is zero along $\gamma(t)$, and by the Hopf--Rinow theorem $f$ vanishes in $M$, which contradicts the assumption that $f$ is not constant.
 Thus ${\rm d}f$ cannot vanish on $f^{-1}(0)$, and $0$ is a regular value of $f$,
 which implies that $f^{-1}(0)$ is an $(n-1)$-dimensional submanifold of $M$. Hence, ${\rm d}R_\phi^m=0$ on an open dense set and hence everywhere in $M$.
\end{proof}

Now are ready to prove Proposition \ref{propRneg}.

\begin{proof}[Proof of Proposition \ref{propRneg}]
It is well known that if a closed smooth metric measure space satisfies ${\rm Rc}_\phi^m=0$, then it is Ricci-flat and $\phi$ is constant \cite{KimKim}. So we can assume that $\bigl(M,g,{\rm e}^{-\phi}{\rm d}V_g,m\bigr)$ is a nontrivial weighted vacuum static space.
Then it follows from Proposition \ref{Rphiconst} that
the weighted scalar curvature $R_\phi^m$ is constant, while \eqref{wstatic_del} implies that \smash{$-\Delta_\phi-\frac{R_\phi^m}{m+n-1}$} has nontrivial kernel. Thus $R_\phi^m\geq0$. Moreover, if $R_\phi^m=0$, then $f\equiv c$ for some nonzero constant $c$. But then~\eqref{wstatic} gives ${\rm Rc}_\phi^m=0$ which implies that $\bigl(M,g,{\rm e}^{-\phi}{\rm d}V_g,m\bigr)$ is Ricci-flat and $\phi$ is constant.
\end{proof}

The following corollary is a direct consequence of Proposition \ref{propRneg}.

\begin{Corollary}\label{cor4.2}
Let $\bigl(M,g,{\rm e}^{-\phi}{\rm d}V_g,m\bigr)$ be a connected closed smooth metric measure space
such that its weighted scalar curvature $R_\phi^m$ is zero
and $(M,g)$ is not Ricci-flat.
Then $\bigl(M,g,{\rm e}^{-\phi}{\rm d}V_g,m\bigr)$
is not a weighted vacuum static space.
\end{Corollary}

By implicit function theorem, we have the following.
\begin{Proposition}\label{prop3.6}
 Let $\bigl(M,g,{\rm e}^{-\phi}{\rm d}V_g,m\bigr)$ be a connected closed smooth metric measure space. Assume that either $(i)$ $R_\phi^m$ is not equal to $\lambda(n+m-1)$ where $\lambda\in spec(-\Delta_\phi)$, or $(ii)$ $R_\phi^m=0$ and ${\rm Rc}_\phi^m\not\equiv 0$. Then $\mathcal{DR}_{(g,\phi)}$ is surjective, its kernel splits and $\mathcal{R}$ maps any neighborhood of~$(g,\phi)$ onto a neighborhood of $R_\phi^m$.
\end{Proposition}
\begin{proof}
 Since $\mathcal{DR}_{(g,\phi)}^*$ has injective symbol, by Berger--Ebin splitting theorem \cite{Berger&Ebin}, it suffices to show that $\mathcal{DR}_{(g,\phi)}^*$ is injective, for then $\mathcal{DR}_{(g,\phi)}$ will be surjective and its kernel will have a closed complement, namely \smash{$\mbox{Im}\bigl(\mathcal{DR}_{(g,\phi)}^*\bigr)$,} the image of $\mathcal{DR}_{(g,\phi)}^*$. Injectivity of $\mathcal{DR}_{(g,\phi)}^*$ follows from Proposition \ref{propRneg}. The local surjectivity of $\mathcal{R}$ then follows by the implicit function theorem.
\end{proof}

We have the following example of
 nontrivial weighted vacuum static space.

\begin{Example}
 Consider $\bigl(\mathbb{R}^n,g_0,{\rm e}^{-\phi}{\rm d}V_{g_0},m\bigr)$, where $g_0$ is the flat metric in $\mathbb{R}^n=\{(x_1,\dots,x_n)\mid \allowbreak x_i\in\mathbb{R}\}$. Suppose that $f=f(x_1)$ and $\phi=\phi(x_1)$. Then \eqref{wstatic} can be reduced to the following two ODEs:
 \begin{gather}\label{rnode}
 \phi'f'-f\phi''+\frac{1}{m}f\bigl(\phi'\bigr)^2=0,\\
f''-\frac{m-1}{m}\phi'f'+\frac{1}{m}f\phi''-\frac{1}{m}f\bigl(\phi'\bigr)^2=0.
 \end{gather}
Then one can easily see that $(f,\phi)=\bigl({\rm e}^{-\frac{1}{m}x_1},x_1\bigr)$ satisfies \eqref{rnode}. So $\bigl(\mathbb{R}^n,g_0,{\rm e}^{-x_1}{\rm d}V_{g_0},m\bigr)$ is a nontrivial weighted vacuum static space with \smash{$f={\rm e}^{-\frac{1}{m}x_1}$} and the weighted scalar curvature $R_\phi^m=-\frac{m+1}{m}$.
In particular, this example shows that the assumption that $M$ is compact is needed in Proposition \ref{propRneg}.
\end{Example}

The following proposition gives an existence result
for prescribing the weighted scalar curvature.

\begin{Proposition}\label{prop4.2}
Let $\bigl(M,g,{\rm e}^{-\phi}{\rm d}V_g,m\bigr)$ be a closed smooth metric measure space
such that its weighted scalar curvature $R_\phi^m$ is zero
and $(M,g)$ is not Ricci-flat.
For any $f\in C^\infty(M)$,
there exists a
smooth metric measure space \smash{
$\bigl(M,\overline{g},{\rm e}^{-\overline{\phi}}{\rm d}V_{\overline{g}},m\bigr)$} such that
its weighted scalar curvature~$R_{\overline{\phi}}^m$ is equal to $f$.
\end{Proposition}
\begin{proof}\looseness=-1
It follows from Corollary \ref{cor4.2} that
$\bigl(M,g,{\rm e}^{-\phi}{\rm d}V_g,m\bigr)$ is not a weighted vacuum static space.
In particular, $\mathcal{DR}_{(g,\phi)}^*$ is injective. Applying Proposition \ref{prop3.6}, there exists $\epsilon>0$ such that~if%
\begin{equation}\label{condition}
\|\psi\|=\|R_\phi^m-\psi\|<\epsilon,
\end{equation}
then there exists a smooth metric measure space $\bigl(M,\overline{g},{\rm e}^{-\overline{\phi}}{\rm d}V_{\overline{g}},m\bigr)$
such that its weighted scalar curvature $R_{\overline{\phi}}^m=\psi$.
In particular, for any $f\in C^\infty(M)$, we can find a constant $c>0$ such that
\[\frac{\|f\|}{c}<\epsilon,\]
where $\epsilon$ is the number appeared in (\ref{condition}).
It follows from (\ref{condition}) that
there exists a smooth metric measure space $\bigl(M,\overline{g},{\rm e}^{-\overline{\phi}}{\rm d}V_{\overline{g}},m\bigr)$
such that its weighted scalar curvature \smash{$R_{\overline{\phi}}^m=\frac{f}{c}$.}
Hence, the smooth metric measure space $\bigl(M,c^{-1}\overline{g},c^{-\frac{n+m}{2}}{\rm e}^{-\overline{\phi}}{\rm d}V_{\overline{g}},m\bigr)$
has weighted scalar curvature $cR_{\overline{\phi}}^m=f$.
\end{proof}

Suppose $(M,g_M)$ is an $n$-dimensional compact space form with sectional curvature $1$,
and $(N,g_N)$ is an $n$-dimensional compact space form with sectional curvature $-1$.
Then $M\times N$ equipped with the product metric $g_M\oplus g_N$
is not Ricci-flat. Moreover,
the smooth metric measure space $(M\times N,g_M\oplus g_N,{\rm d}V_{g_M\oplus g_N},m)$, i.e., $\phi\equiv 0$,
has zero weighted scalar curvature. It follows from Proposition \ref{prop4.2} that
we can find
$\bigl(M\times N,\overline{g},{\rm e}^{-\overline{\phi}}{\rm d}V_{\overline{g}},m\bigr)$ such that
its weighted scalar curvature $R_{\overline{\phi}}^m$ is equal to any prescribed smooth function $f$.

\section{Weighted vacuum static space} \label{sec4}

In general relativity, a \textit{static spacetime} is a four-dimensional Lorentzian manifold which possesses a timelike Killing field and a spacelike hypersurface which is orthogonal to the integral curves of this Killing field. In this case coordinates can be chosen so that the metric $\overline{g}$ is a warped product of the hypersurface (with metric $g$) and a time interval, where the warping factor $f$ is independent of time, i.e.,
\begin{equation*}
 \overline{g}=-f^2{\rm d}t^2+g.
\end{equation*}
The next proposition shows that the weighted vacuum static space is actually related to static spacetimes in the weighted sense.

\begin{Proposition}
 Let $\bigl(M,g,{\rm e}^{-\phi}{\rm d}V_g,m\bigr)$ be a smooth metric measure space.
 Consider $\bigl(\mathbb{R}\times M,\allowbreak \overline{g}=-f^2{\rm d}t^2+g,{\rm e}^{-\overline{\phi}}{\rm d}V_{\overline{g}},m\bigr)$ as a smooth metric measure space where $\overline{\phi}$ is the pullback of~$\phi$ via the projection $\mathbb{R}\times M\to M$. Then $\bigl(M,g,{\rm e}^{-\phi}{\rm d}V_g,m\bigr)$ is a weighted vacuum static space with a potential function $f$ if and only if the warped product metric $\overline{g}=-f^2{\rm d}t^2+g$ is Einstein whenever $f\neq0$ in the weighted sense, i.e., ${\rm Rc}\,_{\overline{\phi}}^m(\overline{g})=k\overline{g}$ for some constant $k$.
\end{Proposition}
\begin{proof}
 We work on a component of the open set where $f$ is nowhere-zero. Then
 we have the following well-known formulas for the curvature tensor of the warped product $\bigl(B\times_f F,\overline{g}=-f^2{\rm d}t^2+g\bigr)$:
 For vectors $X,Y$ tangent to the base $B$ and $V$, $W$ tangent to the $1$-dimensional fiber $F$, and with ${\rm \overline{Rc}}={\rm Rc}(\overline{g})$,
 \begin{gather*}
 {\rm \overline{Rc}}(X,Y)={\rm Rc}^B(X,Y)-\frac{1}{f}{\rm Hess}_gf(X,Y),\\
 {\rm \overline{Rc}}(X,V)=0,\\
 {\rm \overline{Rc}}(V,W)=-\frac{\Delta_gf}{f}\overline{g}(V,W).
 \end{gather*}
 Using these, we compute the Bakry--\'{E}mery Ricci tensor of $\overline{g}$ as follows:
 \begin{gather*}
 {\rm \overline{Rc}}\,_{\overline{\phi}}^m(X,Y)={\rm \overline{Rc}}(X,Y)+{\rm Hess}\,_{\overline{g}}\overline{\phi}(X,Y)-\frac{1}{m}{\rm d}\overline{\phi}\otimes {\rm d}\overline{\phi}(X,Y)\\
 \hphantom{{\rm \overline{Rc}}\,_{\overline{\phi}}^m(X,Y)}={\rm Rc}^B(X,Y)-\frac{1}{f}{\rm Hess}_gf(X,Y)+{\rm Hess}_{g}\phi(X,Y)-\frac{1}{m}{\rm d}\phi\otimes {\rm d}\phi(X,Y)\\
 \hphantom{{\rm \overline{Rc}}\,_{\overline{\phi}}^m(X,Y)}={\rm Rc}_\phi^m(X,Y)-\frac{1}{f}{\rm Hess}_gf(X,Y)\\
 \hphantom{{\rm \overline{Rc}}\,_{\overline{\phi}}^m(X,Y)}=-\frac{\Delta_\phi f}{f}g(X,Y)\\
 \hphantom{ {\rm \overline{Rc}}\,_{\overline{\phi}}^m(X,Y)}=\frac{R_\phi^m}{n+m-1}g(X,Y)=\frac{R_\phi^m}{n+m-1}\overline{g}(X,Y),
 \end{gather*}
 where the fourth equality follows from \eqref{wstatic} and the fifth equality follows from \eqref{wstatic_del}.
Similarly, one can easily see that
\begin{gather*}
 {\rm \overline{Rc}}\,_{\overline{\phi}}^m(X,V)=0,\\
 {\rm \overline{Rc}}\,_{\overline{\phi}}^m(V,W)=-\frac{\Delta_\phi f}{f}\overline{g}(V,W)=\frac{R_\phi^m}{n+m-1}\overline{g}(V,W).
\end{gather*}
Therefore, the warped product metric $\overline{g}$ is Einstein in the weighted sense.

The converse can be proved in a similar way.
\end{proof}

We say that $(M,g,f)$ is a \textit{vacuum static space}, if $(M,g)$ is a Riemannian manifold and smooth function $f\neq0$ satisfies the following equation:
\begin{equation}\label{vacuum}
 {\rm Hess}_g f=f\left({\rm Rc}-\frac{R}{n-1}g\right).
\end{equation}

\begin{Proposition}
Let $m>0$. Let $\mathcal{M}_R$ be the space of all closed, connected vacuum static spaces and $\mathcal{M}_w$ be the space of all closed,
connected weighted vacuum static spaces. Suppose $\bigl(M,g,{\rm e}^{-\phi}{\rm d}V_g,f\bigr)\in \mathcal{M}_R\cap \mathcal{M}_w$ $($with the same $f)$, then $\bigl(M,g,{\rm e}^{-\phi}{\rm d}V_g,m\bigr)$ is isometric to a~Ricci-flat manifold with $\phi$ being constant.
\end{Proposition}
\begin{proof}
Suppose $f^{-1}(0)\neq \varnothing$. It is well known that in a vacuum static space, ${\rm d}f\neq0$ on $f^{-1}(0)$ (cf. \cite{Fischer&Marsden}). Then $f^{-1}(0)$ is a regular hypersurface in $M$ and hence $f\neq0$ on a dense subset of $M$. Combining \eqref{wstatic2} and \eqref{vacuum}, we have
 \begin{equation*}
 -\frac{R}{n-1}f=\Delta_gf=f\left(\frac{m-1}{m+n-1}R_\phi^m-\Delta_\phi\phi\right).
 \end{equation*}
Since $f$ is nonzero on a dense set, this implies that
\begin{equation}\label{delphiphi}
 \Delta_\phi\phi=\frac{m-1}{m+n-1}R_\phi^m+\frac{R}{n-1}.
\end{equation}
 It is well known that if $(M,g,f)$ is vacuum static space, then $(M,g)$ has nonnegative constant scalar curvature. And by Proposition \ref{Rphiconst}, the weighted scalar curvature $R_\phi^m$ is also nonnegative constant. If $m\geq1$, then \eqref{delphiphi} implies $\Delta_\phi\phi\geq0$. Therefore, $R=0$ and $\phi$ is constant. Thus $R_\phi^m=0$. We conclude from Proposition \ref{propRneg} that $\bigl(M,g,{\rm e}^{-\phi}{\rm d}V_g,m\bigr)$ is Ricci-flat. Suppose $0<m<1$. By \eqref{wsc}, we can rewrite \eqref{delphiphi} as
 \begin{equation}\label{delphiphi2}
 \Delta_\phi\phi=\frac{mn}{(n-1)(m+n-1)}R+\frac{2(m-1)}{m+n-1}\Delta \phi-\frac{(m-1)(m+1)}{m(m+n-1)}|\nabla \phi|^2.
 \end{equation}
 Integrating \eqref{delphiphi2} over ${\rm d}V_g$, we have
 \begin{equation*}
 \int_M \Delta_\phi\phi {\rm d}V_g=\frac{mn}{(n-1)(m+n-1)}\int_MR{\rm d}V_g-\frac{(m-1)(m+1)}{m(m+n-1)}\int_M|\nabla \phi|^2{\rm d}V_g\geq0.
 \end{equation*}
 But based on the formula for $\Delta_\phi\phi$, this implies that $-\int_M|\nabla \phi|^2{\rm d}V_g\geq0$. Therefore, $\phi$ is constant and then $R=0$. Thus $R_\phi^m=0$. Again by Proposition \ref{propRneg}, we conclude that $\bigl(M,g,{\rm e}^{-\phi}{\rm d}V_g,m\bigr)$ is Ricci-flat.
\end{proof}

\begin{Remark}
 It would be interesting to study the relation between $\mathcal{M}_R$ and $\mathcal{M}_w$. One can ask the following question: If both $\mathcal{M}_R$ and $\mathcal{M}_w$ are nonempty, does that mean $\mathcal{M}_R\cap\mathcal{M}_w$ is also nonempty? This question is left to future research, and we hope to come back to this later.
\end{Remark}

 We say that $\bigl(M,g,{\rm e}^{-\phi}{\rm d}V_g,m\bigr)$ is \textit{locally conformally flat in the weighted sense} if for each $p\in M$, there is a neighborhood $U$ of $p$ on which there is a conformal factor $u$ for which $\bigl({\rm e}^{2u}g,{\rm e}^{(m+n)u}{\rm e}^{-\phi}{\rm d}V_g\bigr)=(g_{\rm flat},{\rm d}V_{g_{\rm flat}})$. Just as the local conformal flatness in the Riemannian manifold $\bigl(M^n,g\bigr)$ is equivalent to the vanishing of the Weyl tensor when $n\geq4$, the local conformal flatness in the weighted sense
has the following characterization:
 Consider the following modifications of the curvature tensors:
\begin{gather*}
P_\phi^m:={\rm Rc}^\phi-\frac{1}{2(m+n-1)}R_\phi^mg,\\
A_\phi^m:={\rm Rm}-\frac{1}{m+n-2}P_\phi^m\owedge g.
\end{gather*}
Here $\owedge$ denoted the Kulkarni--Nomizu product. We call $P_\phi^m$ the \textit{weighted Schouten tensor}, and~$A_\phi^m$ the \textit{weighted Weyl tensor}. Then a smooth metric measure space $\bigl(M,g,{\rm e}^{-\phi}{\rm d}V_g,m\bigr)$, where $n\geq3$ and $m+n\neq3$, is locally conformally flat in the weighted sense if and only if $A_\phi^m=0$
(cf.~\cite[Lemma 6.6]{Case1}).

We are now ready to prove Theorem \ref{thm1.1}.

\begin{proof}[Proof of Theorem \ref{thm1.1}]
We follow the argument in \cite{Catino} (see also \cite{He}).
Taking the covariant derivative of \eqref{wstatic2} yields
 \begin{equation*}
 \nabla_i\nabla_j\nabla_kf={\rm Rc}^\phi_{jk}\nabla_if+f\nabla_i{\rm Rc}^\phi_{jk}-\frac{R_\phi^m}{n+m-1}g_{jk}\nabla_if,
 \end{equation*}
 where ${\rm Rc}^\phi_{ij}$ denotes ${\rm Rc}_\phi^m({\rm e}_i,{\rm e}_j)$. This implies that
 \begin{align}\label{wstatic_rijkl}
 R_{ijkl}\nabla^l f& = \nabla_i\nabla_j\nabla_kf-\nabla_j\nabla_i\nabla_kf\\
& = {\rm Rc}^\phi_{jk}\nabla_if-{\rm Rc}^\phi_{ik}\nabla_jf+f\nabla_i{\rm Rc}^\phi_{jk}-f\nabla_j{\rm Rc}^\phi_{ik}
 -\frac{R_\phi^m}{n+m-1}(g_{jk}\nabla_if-g_{ik}\nabla_jf).\nonumber
 \end{align}

Note that if $\bigl(M,g,{\rm e}^{-\phi}{\rm d}V_g,m\bigr)$ is locally conformally flat in the weighted sense, then the weighted Schouten tensor
\[P_\phi^m:={\rm Rc}_\phi^m-\frac{R_\phi^m}{2(m+n-1)}g\] is a Codazzi tensor, i.e., $\nabla_iP^\phi_{jk}=\nabla_jP^\phi_{ik}$ (see \cite[Lemma 3.2]{Case2}). Since $R_\phi^m$ is constant by Proposition \ref{propRneg}, this implies that the Bakry--\'{E}mery Ricci tensor is a Codazzi tensor.
So \eqref{wstatic_rijkl} reduces to\looseness=-1
\begin{equation}\label{rijkl2}
 R_{ijkl}\nabla^l f= {\rm Rc}^\phi_{jk}\nabla_if-{\rm Rc}^\phi_{ik}\nabla_jf-\frac{R_\phi^m}{n+m-1}(g_{jk}\nabla_if-g_{ik}\nabla_jf).
 \end{equation}
On the other hand, it follows from $A_\phi^m\equiv 0$ that
\begin{align}\label{rijkl3}
 \nonumber R_{ijkl}\nabla^lf={}&-\frac{1}{m+n-2}\bigl({\rm Rc}^\phi_{il}g_{jk}\nabla^lf+{\rm Rc}^\phi_{jk}\nabla_if-{\rm Rc}_{ik}^\phi\nabla_jf-{\rm Rc}^\phi_{jl}g_{ik}\nabla^lf\bigr)\\
&+\frac{R_\phi^m}{(n+m-1)(n+m-2)}(g_{jk}\nabla_if-g_{ik}\nabla_jf).
\end{align}
Combining \eqref{rijkl2} and \eqref{rijkl3}, we have
\begin{gather}
\nonumber -{\rm Rc}_{il}^\phi g_{jk}\nabla^lf +{\rm Rc}_{jl}^\phi g_{ik}\nabla^lf\\
\qquad{} =(n+m-1)\bigl({\rm Rc}_{jk}^\phi\nabla_if-{\rm Rc}_{ik}^\phi\nabla_jf\bigr)-R_\phi^m(g_{jk}\nabla_if-g_{ik}\nabla_jf).\label{rijkl4}
\end{gather}
 This shows that $\nabla f$ is an eigenvector of ${\rm Rc}^\phi$, i.e., ${\rm Rc}^\phi(X,\nabla f)=0$ for $X\perp \nabla f$.

For any regular value $c_0$ of the function $f$, consider the level surface $\Sigma_{c_0}=f^{-1}(c_0)$. Suppose~$I$ is an open interval containing $c_0$ such that $f$ has no critical points in the open neighborhood $U_I=f^{-1}(I)$ of $\Sigma_{c_0}$. Then we can express the metric $g$ on $U_I$ as
\begin{equation*}
 {\rm d}s^2=\frac{1}{|\nabla f|^2}{\rm d}f^2+g_{f},
\end{equation*}
where $g_{f}=g_{ab}(f,\theta){\rm d}\theta^a{\rm d}\theta^b$ is the induced metric and $\theta=\bigl(\theta^1,\dots,\theta^n\bigr)$ is any local coordinates system on $\Sigma_{c_0}$.

On the other hand, for any vector field $X\perp \nabla f$, we have
\begin{equation*}
 \nabla_X\bigl(|\nabla f|^2\bigr)=2\nabla^2f(\nabla f,X)=2f\left({\rm Rc}^\phi(\nabla f,X)-\frac{R_\phi^m}{n+m-1}g(\nabla f,X)\right)=0,
\end{equation*}
where the second equality follows from (\ref{wstatic2})
and the last equality follows from the fact that $\nabla f$ is an eigenvector of ${\rm Rc}^\phi$. Hence, $|\nabla f|^2$ is constant on any regular value surface $\Sigma_c=f^{-1}(c)\subset U_I$, which are all diffeomorphic to $\Sigma_{c_0}$. This allows us to make the change of variable by setting
\begin{equation*}
 r(x)=\int \frac{{\rm d}f}{|\nabla f|},
\end{equation*}
so that we can further express the metric $g$ on $U_I$ as
\begin{equation*}
 {\rm d}s^2={\rm d}r^2+g_{ab}(r,\theta){\rm d}\theta^a{\rm d}\theta^b.
\end{equation*}
Let $\nabla r=\frac{\partial}{\partial r}$. Then $|\nabla r|=1$ and $\nabla f=f'(r) \frac{\partial }{\partial r}$ on $U_I$. Note that $f'(r)$ does not change sign on~$U_I$ because $f$ has no critical points there. Thus, we may assume $I=(\alpha,\beta)$ with $f'(r)>0$ for $r\in (\alpha,\beta)$. It is also easy to check that
\begin{equation}\label{cao2.4}
 \nabla_{\frac{\partial}{\partial r}}\frac{\partial}{\partial r}=0,
\end{equation}
so integral curves to $\nabla r$ are geodesics.

Next, it follows from \eqref{wstatic2} and \eqref{cao2.4} that
\begin{equation}\label{cao2.4b}
 f\left({\rm Rc}^\phi_{rr}-\frac{R_\phi^m}{n+m-1}\right)=\nabla^2f\left(\frac{\partial}{\partial r},\frac{\partial}{\partial r}\right)=f''(r).
\end{equation}
We can therefore conclude that ${\rm Rc}^\phi_{rr}$ is also constant on $\Sigma_c\subset U_I$. Moreover, the second fundamental form of $\Sigma_c$ is given by
\begin{equation}\label{cao2.6}
 h_{ab}=\frac{\nabla_a\nabla_bf}{|\nabla f|}=\frac{f}{f'}\left({\rm Rc}^\phi_{ab}-\frac{R_\phi^m}{n+m-1}g_{ab}\right).
\end{equation}
On the other hand, by \eqref{rijkl4}, we have
\begin{equation}\label{cao2.6b}
 {\rm Rc}^\phi_{ab}=\frac{1}{n+m-1}\bigl(R_\phi^m-{\rm Rc}^\phi_{rr}\bigr)g_{ab}.
\end{equation}
Combining \eqref{cao2.6} and \eqref{cao2.6b}, we have
\begin{equation}\label{cao2.6c}
 h_{ab}=-\frac{1}{n+m-1}\frac{f}{f'}{\rm Rc}^\phi_{rr}g_{ab}.
\end{equation}
In particular, $\Sigma_c$ is umbilical and its mean curvature is given by
\begin{equation*}
 H=-\frac{n-1}{n+m-1}\frac{f}{f'}{\rm Rc}^\phi_{rr}g_{ab},
\end{equation*}
which is again constant along $\Sigma_c$.

Now, we fix a local coordinates system
\begin{equation*}
 \bigl(x^1,x^2,\dotsc,x^n\bigr)=\bigl(r,\theta^2,\dotsc,\theta^n\bigr)
\end{equation*}
in $U_I$, where $\bigl(\theta^2,\dotsc,\theta^n\bigr)$ is any local coordinates system on the level surface $\Sigma_{c_0}$, and indices $a,b,c,\dotsc$ range from $2$ to $n$. Then, computing in this local coordinates system we obtain that
\begin{equation*}
 h_{ab}=-\langle \partial_r,\nabla_a\partial_b\rangle=-\bigr\langle \partial_r,\Gamma_{ab}^1\partial_r\bigl\rangle =-\Gamma_{ab}^1.
\end{equation*}
But the Christoffel symbol $\Gamma_{ab}^1$ is given by
\begin{equation*}
 \Gamma_{ab}^1=\frac{1}{2}g^{11}\left(-\frac{\partial g_{ab}}{\partial r}\right)=-\frac{1}{2}\frac{\partial g_{ab}}{\partial r}.
\end{equation*}
Hence, we get
\begin{equation*}
 \frac{\partial g_{ab}}{\partial r}=-\frac{2}{n+m-1}\frac{f}{f'}{\rm Rc}^\phi_{rr}g_{ab}=-\frac{2}{n+m-1}\frac{f}{f'}\left(\frac{f''}{f}+\frac{R_\phi^m}{n+m-1}\right)g_{ab},
\end{equation*}
where the second equality follows from \eqref{cao2.4b}.
Thus we can see that in any open neighborhood $U_{\alpha}^\beta=f^{-1} ((\alpha,\beta) )$ of $\Sigma_c$ in which $f$ has no critical points, the metric $g$ can be expressed as
\begin{equation*}
 {\rm d}s^2={\rm d}r^2+w(r)^2\bar{g},
\end{equation*}
where $\bigl(\theta^2,\dots,\theta^n\bigr)$ is any local coordinates system on $\Sigma_c$, $\bar{g}=g_{ab}(r_0,\theta){\rm d}\theta^a{\rm d}\theta^b$ is the induced metric on $\Sigma_c=r^{-1}(r_0)$, and the warping function $w(r)$ satisfying
\begin{equation*}
 \frac{w'}{w}=-\frac{1}{n+m-1}\frac{f''}{f'}-\frac{R_\phi^m}{(n+m-1)^2}\frac{f}{f'}.
\end{equation*}
Furthermore, from the Gauss equation, one can see that the sectional curvatures of $(\Sigma_c,\bar{g})$ are given by
\begin{align*}
 R_{abab}^\Sigma&= R_{abab}+h_{aa}h_{bb}-h_{ab}^2\\
& = \frac{1}{n+m-2}\bigl({\rm Rc}^\phi_{aa}+{\rm Rc}^\phi_{bb}\bigr)-\frac{R_\phi^m}{(n+m-1)(n+m-2)}+\frac{1}{(n+m-1)^2}\left(\frac{f}{f'}Rc^\phi_{rr}\right)^2\\
& = \frac{R_\phi^m-2{\rm Rc}^\phi_{rr}}{(n+m-1)(n+m-2)}+\frac{1}{(n+m-1)^2}\left(\frac{f}{f'}{\rm Rc}^\phi_{rr}\right)^2
\end{align*}
for $a,b=2,\dots,n$, where the second equality follows from $A_\phi^m\equiv 0$ and \eqref{cao2.6c},
 and the third equality follows from \eqref{cao2.6b}. Since all the terms on the right-hand side are constant on $\Sigma_c$, we conclude that the sectional curvatures of $(\Sigma_c,\bar{g})$ are constant.
\end{proof}

\section[Rigidity phenomena of flat manifolds and prescribing the weighted scalar curvature]{Rigidity phenomena of flat manifolds\\ and prescribing the weighted scalar curvature}\label{sec5}

First we prove Theorem \ref{rigidity}.

\begin{proof}[Proof of Theorem \ref{rigidity}]
 Let $(F,g_F)$ be a $m$-dimensional closed flat Riemannian manifold. Since $\overline{g}$ is flat and $\overline{\phi}$ is constant,
 \smash{$\bigl(M\times F,g_1:=\overline{g}+{\rm e}^{-2\overline{\phi}/m}g_F\bigr)$} is a $(n+m)$-dimensional closed flat Riemannian manifold. By \cite[Theorem~B]{Fischer&Marsden}, there is a $\epsilon_1$ such that if a Riemannian metric~$h$ on $M\times F$ has nonnegative scalar curvature and $\|g_1-h\|_{C^2(M\times F,g_1)}<\epsilon_1$, then $h$ is flat. Let~$(g,\phi)$ be a smooth metric measure structure on $M$ with $R_\phi^m\geq0$. \smash{Then $g_2:=g+{\rm e}^{-\frac{2\phi}{m}}g_F$ is} a Riemannian metric on $M\times F$ such that its scalar curvature $R_{g_2}=R_{\phi}^m$ is nonnegative. And one can easily see that
 \begin{equation*}
 \|g_1-g_2\|_{C^2(M\times F,g_1)}\leq C_1\bigl\|\bigl(\overline{g},\overline{\phi}\bigr)-(g,\phi)\bigr\|_{C^2(M,\overline{g})}
 \end{equation*}
 for some constant $C_1=C_1\bigl(M,\overline{g},\overline{\phi}\bigr)$. Thus if $\bigl\|\bigl(\overline{g},\overline{\phi}\bigr)-(g,\phi)\bigr\|_{C^2(M,\overline{g})}<\frac{\epsilon_1}{C_1}$ then we have $\|g_1-g_2\|_{C^2(M\times F,g_1)}<\epsilon_1$ which implies that \smash{$g_2=g+{\rm e}^{-\frac{2\phi}{m}}g_F$} must be flat. Since $g_F$ is flat, this implies that $g$ is flat and $\phi$ is constant.
\end{proof}

In the remainder of this section, we consider the problem of prescribing the weighted scalar curvature on smooth metric measure spaces.
More precisely, given a smooth function $f$ in $M$,
we want to find a smooth metric measure space $\bigl(M, g, {\rm e}^{-\phi}{\rm d}V_g, m\bigr)$ such that the weighted scalar curvature is equal to $f$.
In particular, we are going to prove Theorem \ref{thm1.3}.
To this end, we first recall the following approximation lemma.

\begin{Lemma}[{\cite[Theorem 2.1]{Kazdan&Warner2}}]\label{approxlem}
 Let $f_1, f_2\in C^\infty(M)$. If $\min f_1\leq f_2\leq \max f_1$ in $M$, then given any positive $\epsilon$, there is a diffeomorphism $\varphi$ of $M$ such that, for $p>n$, we have that
\begin{equation*}
 \|f_1\circ\varphi-f_2\|_{L^p(M)}<\epsilon.
\end{equation*}
\end{Lemma}

We also need the following.

\begin{Proposition}\label{prop6.2}
 Let $f\in L^p(M)$ with $p>n$. Suppose that $\mathcal{DR}_{(g_0,\phi_0)}^*$ is injective.
 There is an~$\eta>0$ such that if
\begin{equation*}
 \|f-R_{\phi_0}^m\|_{L^p(M)}<\eta,
\end{equation*}
then there is a metric measure $(g_1,\phi_1)\in \mathcal{M}^{2,p}\times W^{2,p}$ such that $\mathcal{R}(g_1,\phi_1)=f$. Moreover,~$(g_1,\phi_1)$ is smooth in any open set where $f$ is smooth.
\end{Proposition}
\begin{proof}
 We consider the following operator $S\colon U\subset W^{4,p}(M)\to L^p(M)$ defined by
\begin{equation*}
 S(u)=R_{\overline{\phi}}^m,
\end{equation*}
where
\begin{equation*}
 \bigl(\overline{g},\overline{\phi}\bigr)=(g_0,\phi_0)+\mathcal{DR}_{(g_0,\phi_0)}^*(u)
\end{equation*}
and $U$ is a sufficiently small neighborhood of zero in $W^{4,p}$. We claim that $S'(0)$ is an isomorphism restricted to a small neighborhood in $W^{4,p}$ norm. In fact, $S(0)=\mathcal{R}(g_0,\phi_0)$ and
\begin{equation*}
 S'(0)v=\mathcal{DR}_{(g_0,\phi_0)}\mathcal{DR}^*_{(g_0,\phi_0)}v.
\end{equation*}
Hence, $\ker S'(0)=\ker \mathcal{DR}_{(g_0,\phi_0)}\mathcal{DR}^*_{(g_0,\phi_0)}\subseteq\ker \mathcal{DR}^*_{(g_0,\phi_0)}=\{0\}$,
which implies $\ker S'(0)=\{0\}$. It follows from the implicit function theorem that $S$ maps a neighborhood of zero in $W^{4,p}$ onto a neighborhood $S(0)=\mathcal{R}(g_0,\phi_0)$ in $L^p(M)$. Thus there is an $\eta>0$ such that if
\begin{equation*}
 \|f_1-R_{\phi_0}^m\|_{L^p(m)}<\eta,
\end{equation*}
then there exists a solution $(g_1,\phi_1)=(g_0,\phi_0)+\mathcal{DR}^*(u)$ of $\mathcal{R}(g_1,\phi_1)=f$. Using elliptic regularity and a bootstrap argument, we have that if $f$ is smooth, then $u$ is smooth.
\end{proof}

Now we are ready to prove Theorem \ref{thm1.3}.

\begin{proof}[Proof of Theorem \ref{thm1.3}]
We perturb $(g_0,\phi_0)$ slightly, if necessary, to obtain $(g_1,\phi_1)$
such that the weighted scalar curvature $\mathcal{R}(g_1,\phi_1)$ is not constant
with $c_1\min K< \mathcal{R}(g_1,\phi_1)<c_1\max K$ in $M$, where $c_1>0$ is constant. It follows from Proposition \ref{propRneg} that $\bigl(M,g_1,{\rm e}^{-\phi_1}{\rm d}V_{g_1},m\bigr)$ is not a weighted vacuum static space.
In particular, $\mathcal{DR}_{(g_1,\phi_1)}^*$ is injective. It follows from Lemma~\ref{approxlem} that there is a diffeomorphism $\varphi$ of $M$ such that $\|c_1K\circ\varphi-k_1\|_p<\eta$, where $p>\dim M$. Since~$\mathcal{DR}_{(g_1,\phi_1)}^*$ is injective,
we can apply Proposition \ref{prop6.2} to conclude that
there is a smooth metric measure $(g_2,\phi_2)$ such that $\mathcal{R}(g_2,\phi_2)=c_1K\circ \varphi$.
If we let $(g,\phi)=\bigl(\varphi^{-1}\bigr)^* (c_1(g_2,\phi_2) )$,
the smooth metric measure space $\bigl(M,g,{\rm e}^{-\phi}{\rm d}V_g,m\bigr)$ has weighted scalar curvature being equal to~$K$.
\end{proof}

\subsection*{Acknowledgements}
The authors would like to thank the referees for comments and suggestions, which improve the presentation of this paper.
The first author was supported by the National Science and Technology Council (NSTC),
Taiwan, with grant Number: 112-2115-M-032-006-MY2,
and the second author was supported by a KIAS Individual Grant (SP070701) via the Center for Mathematical Challenges at Korea Institute for Advanced Study.

\pdfbookmark[1]{References}{ref}
\LastPageEnding

\end{document}